\documentclass[a4paper,12pt]{article}
\usepackage{geometry,latexsym,amssymb}
\usepackage[latin5]{inputenc}
\usepackage{enumerate}
\usepackage{enumitem}
\usepackage[usenames]{color}
\usepackage{graphicx}
\usepackage{enumerate}
\usepackage{rotating}
\usepackage{multirow}
\usepackage{cite}
\usepackage{amsthm}
\usepackage{tikz}
\usetikzlibrary{matrix,arrows}
\usepackage{amsmath,amscd}
\usepackage{calc}
\usepackage{ifthen}
\usepackage[hang,labelsep=period]{caption}
\usepackage{fancyhdr}
\usepackage{amssymb,amscd,amsthm, verbatim,amsmath,color,fancyhdr, mathrsfs}
\usepackage{graphicx}
\usepackage{turnstile}
\usepackage[colorinlistoftodos]{todonotes}

\newtheorem{thm}{Theorem}[section]

\newtheorem{lem}[thm]{Lemma}

\newtheorem{ex}[thm]{Example}
\newtheorem{defn}[thm]{Definition}

\begin{document}

\begin{center}
{\Large \textbf{INTEGER REPRESENTATIONS OF CLASSICAL WEYL GROUPS}} \vspace*{0.5cm}
\end{center}

\vspace*{0.3cm}
\begin{center}
HASAN ARSLAN$^{*,1}$, ALNOUR ALTOUM$^{*,2}$, MARIAM ZAAROUR$^{*,3}$ \\
$^{*}${\small {\textit{Department of  Mathematics, Faculty of Science, Erciyes University, 38039, Kayseri, Turkey}}}\\
{\small {\textit{ $^{1}$hasanarslan@erciyes.edu.tr}}}~~{\small {\textit{$^{2}$ alnouraltoum178@gmail.com}}}\\
{\small {\textit{$^{3}$ mariamzaarour94@gmail.com}}}\\[0pt]
\end{center}

\begin{abstract}
In this paper, we define a mixed-base number system over a Weyl group of type $D$, the group even-signed permutations. We introduce one-to-one correspondence between positive integers and elements of Weyl groups of type $D$ after constructed the subexceedant function associating to the group. Thus, the integer representations of all classical Weyl groups are now completed.    
\end{abstract}

\textbf{Keywords}: Even-Signed Permutation, Weyl group, Subexceedant function.\\

\textbf{2020 Mathematics Subject Classification}: 20F55. 
\\

\section{Introduction} 
Integer representation of any element of a classical Weyl group $W$ is a crucial tool to understand the structure of the group and to use efficiently the elements of the group in encryption process. In the case of the symmetric group, Doliskani et al. \cite{br1} introduced a bijection map between positive integers and elements of symmetric groups, Weyl groups of type $A$. They then proposed an effective cryptosystem with the help of this map. When $W$ is of type $B$, the hyperoctahedral base system was described and studied the representation of the elements of Weyl groups of type $B$ by integers by Raharinirina \cite{br4}. Subsequently, some cryptosystems, robust and resistant to attacks by Silver-Pohlig-Helman's algorithm, were established in \cite{br4}. 

Let $S_n$, which is a Weyl group of type $A_{n-1}$, be  the symmetric group of order $n!$. We denote by $B_n$ the group, which is known as the group of signed permutations acting on the set $I_n=\{-n, \cdots, -1, 1, \cdots, n\}$, consists of all bijections $\alpha$ of the set $I_n$ onto itself providing that $\alpha(-i)=-\alpha(i)$ for all $i \in I_n$, where the group operation is the composition of the bijections. Any element $\alpha$ of the group $B_n$ is represented in the window notation such that 
$$\alpha=[s_1 \beta_1, \cdots, s_n \beta_n ]$$
where $\beta \in S_n$,~~$\beta(i)=\beta_i$ and $s_i \in \{-1,1\}$ for all $i \in \{1, \cdots, n\}$. It is well-known that the cardinality of the Weyl group of type $B_n$ is $2^nn!$. As a convention, when multiplying permutations, the rightmost permutations acts first, as usual.

The group of even-signed permutation is denoted by $D_n$, normal subgroup of $B_n$, is the group of signed permutations acting on $I_n$ with an even number of negative entries in their window notation. The size of the Weyl group of type $D_n$ is $2^{n-1}n!$
For further information about the classical Weyl groups see \cite{br2}.

\section{Construction of $D_n$-type Number System}
In this section, we first define the $D_n$-type number system and describe its the structure. 
\begin{defn}
The $D_n$-type number system is a radix base system in which every positive integer $x$ can be expressed in the following form:
\begin{equation}\label{def}
  x=\sum_{i=1}^{n-1} d_iD_i  
\end{equation}
where $d_i \in \{0, 1, 2, \cdots, 2i+1\}$ and $D_i=2^{i-1}i!$ for all the $1 \leq i \leq n-1$. 
\end{defn}
Then, for any positive integer $x$ in the $Dn$-type number system we use the notation 
\[
x=(d_{n-1}:d_{n-2}: \cdots : d_2: d_1)_{D_n}.
\]

There is one-to-one correspondence between positive integers and $D_n$-type number system. 

\begin{thm}\label{exp}
Every positive integer is represented in a unique way in the $D_n$-type base system. 
\end{thm}

Before giving the proof of the theorem, we need the following two lemmas, which involve some fundamental properties of the $D_n$-type number system. These properties have a similar structure to ones of the factoriadic number system of type $A$ and the hyperoctahedral base system of type $B$.

\begin{lem} \label{first}
For any $x=(d_{n-1}:d_{n-2}: \cdots : d_2: d_1)_{D_n}$, we have
\begin{equation}
    0 \leq x \leq D_{n}-1.
\end{equation}
\end{lem}

\begin{proof}
We have an expression 
\[
0 \leq \sum_{i=1}^{n-1} d_iD_i \leq \sum_{i=1}^{n-1} (2i+1) D_i 
\]
due to fact that $d_i \in \{0, 1, 2, \cdots, 2i+1\}$ and $D_i=2^{i-1}i!$ for all $1 \leq i \leq n-1$. Since 
\[
(2i+1) D_i=D_{i+1}-D_i
\]
for each $1 \leq i \leq n-1$, we conclude by direct calculation that $0 \leq \sum_{i=1}^{n-1} d_iD_i \leq D_{n}-1.$
\end{proof}

As a result of Lemma \ref{first}, we can deduce that for any given positive integer n, there are exactly $2^{n-1}n!$ numbers in the $D_n$-type number system.

\begin{lem}\label{second}
Let $x=(d_{n-1}:d_{n-2}: \cdots : d_2: d_1)_{D_n}$ be a number in $D_n$-type number system, then we have
\begin{equation}
    d_{n-1}D_{n-1} \leq x < (d_{n-1}+1)D_{n-1}.
\end{equation}
\end{lem}

\begin{proof}
If we take $y$ as $x-d_{n-1}D_{n-1}=(d_{n-2}: \cdots : d_2: d_1)_{D_n}$, then from Lemma \ref{second}, we can write $y$ in the following way:
\begin{equation}\label{eq}
 0 \leq y \leq D_{n-1}-1.   
\end{equation}
When adding $d_{n-1}D_{n-1}$ to each side of the equation (\ref{eq}), we conclude that the proof is completed, as desired.
\end{proof}
We are now in a position to provide the proof of Theorem \ref{exp}.\\

\textit{$\boldsymbol{Proof~~of~~Theorem~~\ref{exp}}$}: \\
Assume that a positive integer $x$ has two representations in the $D_n$-type number system as follows:
\begin{equation*}
   x=(d_{n-1}:d_{n-2}: \cdots : d_2: d_1)_{D_n}=(e_{m-1}:e_{m-2}: \cdots : e_2: e_1)_{D_n},
\end{equation*}
where $d_{n-1}\neq 0$ and $e_{m-1}\neq 0$. The fact that both $d_{n-1}$ and $e_{m-1}$ are $1$ give rises to
\begin{equation}\label{cntr}
   D_{n-1} \leq d_{n-1} D_{n-1} \leq x ~~~~
  \textrm{and}~~~~D_{m-1} \leq e_{m-1} D_{m-1} \leq x.
\end{equation}
Now we suppose that $n \neq m$. Without loss of generality, we can assume that $n<m$. Then by Lemma \ref{first} and the right side of the equation (\ref{cntr}), we obtain 
\begin{equation*}
    x < D_n \leq D_{m-1} \leq x,
\end{equation*}
which is a contraction. Thus we get $n=m$. 

Now we show that $d_i=e_i$ for all $1 \leq i \leq n-1$. In what follows, we proceed by induction on number of digits. From the equation (\ref{def}), the assertion is clear for $x=(d_1)_{D_n}=(e_1)_{D_n}$. We assume that a positive integer $x$ with k digits in the $D_n$-type number system has a unique representation. Now let $(d_{n-1}:d_{n-2}: \cdots : d_2: d_1)_{D_n}$ and $(e_{n-1}:e_{n-2}: \cdots : e_2: e_1)_{D_n}$ be two representations of $x$ in the $D_n$-type number system. Assume that $d_{n-1}\neq  e_{n-1}$. Without loss of generality, take $d_{n-1}< e_{n-1}$. Thus we get from Lemma \ref{second}
\begin{equation*}
    x < (d_{n-1}+1) D_{n-1} \leq e_{n-1} D_{n-1} \leq x,
\end{equation*}
which leads to a contradiction and hence $d_{n-1}= e_{n-1}$. Since $d_{n-1}= e_{n-1}$ and by the induction hypothesis, the integer $x-d_{n-1}D_{n-1}=x-e_{n-1}D_{n-1}$ has a unique representation and so $d_i=e_i$ for all $1 \leq i \leq n-2$. This completes the proof.

We will now give the method of how to write any positive integer $x$ in the $D_n$ type number system:

The algorithm proceeds in a series of steps. First step starts by dividing $x$ by $4$ and the reminder sets to $r_1=d_1$ in the division process
\begin{equation*}
    x=4q_1+r_1.
\end{equation*}
Then divide $q_1$ by $6$ and the reminder sets to 
$r_2=d_2$ in the following division process
\begin{equation*}
    q_1=6q_2+r_2.
\end{equation*}
If we continue these operations by dividing $q_{i-1}$ by $2(i+1)$ and take $r_i=d_i$ in the expression
\begin{equation*}
    q_{i-1}=2(i+1)q_i+r_i
\end{equation*}
until the quotient $q_t=0$ is zero for some integer $t$.
Eventually, we write the number x as 
\begin{equation}\label{intrep}
   x=(d_{n-1}:d_{n-2}: \cdots : d_2: d_1)_{D_n}
\end{equation}
in $D_n$-type base system.

We illustrate this construction with an example given below:  

\begin{ex}
We choose an integer $x=151100130419$. The expression of integer $x$ in $D_{12}$-type base system is $x=(3:15:6:9:8:5:4:5:7:2:3)_{D_{12}}$.
\end{ex}

\section{Integer Representation of Even-Signed Permutations}

\textit{$\boldsymbol{Subexceedant~~Function~~for~~Weyl~~Group~~of~~type~~D}$} : 

Mantaci and Rakotondrajao \cite{br3} defined subexceedant functions for the symmetric group $S_n$ and showed that there was one-to-one correspondence between permutations and the subexceedant functions. Subexceedant function is a fundamental tool to provide integer representations of the classical Weyl groups, see \cite{br1},\cite{br4}. We will define the subexceedant functions for the group of even-signed permutations by inspiring \cite{br4} and depending on the structure of the group.

\begin{defn}[\cite{br3}]
A subexceedant function $f$ on the set $\{1, \cdots, n\}$ is a map such that 
\begin{equation}
    1 \leq f(i) \leq i,~~\textrm{for}~~\textrm{all}~~1 \leq i \leq n.
\end{equation}
\end{defn}
Denote by $\mathcal{F}_n$ the set of all subexceedant functions on $\{1, \cdots, n\}$ and hence $|\mathcal{F}_n|=n!$. The subexceedant function $f$ on $\{1, \cdots, n\}$ is, in general, expressed by the word $f(1);\cdots; f(n)$. Moreover, the map 
\begin{equation}\label{subex}
   \phi : \mathcal{F}_n \mapsto S_n,~~\phi(f)=(nf(n))\cdots(2f(2))(1f(1)) 
\end{equation}
is a bijection and $(if(i))$ is a transposition for each $1 \leq i \leq n$ \cite{br3}.

Now let $\beta=[\beta_1, \cdots, \beta_n ]$ be an element of $S_n$ given in the window notation. In \cite{br3} Mantaci and Rakotondrajao described the subexceedant function $f$ corresponding to $\beta$ under the map $\phi$ with the following steps:
 \begin{itemize}
     \item Set $f(n)=\beta_n$.
     \item Then take the image of ${\beta}^{-1}(n)$ in the window notation of $\beta$ as $\beta_n$. Thus a new permutation $\beta'$ that contains $n$ as a fixed point is obtained and so $\beta'$ can be thinking as an element of $S_{n-1}$.
     \item Set $f(n-1)=\beta'_{n-1}$.
     \item Continue the same procedure for the permutation $\beta'$ by exchanging the image of ${\beta'}^{-1}(n-1)$ in the window notation of $\beta'$ and $\beta'_{n-1}$ and determine in this manner $f(n-2)$.
     \item Proceed with this iteration until you find all the $f(i)$ values for each $1 \leq i \leq n$.
 \end{itemize}

\begin{defn}
Let $x=(d_{n-1}:d_{n-2}: \cdots : d_2: d_1)_{D_n}$ be a number with the $n-1$ digits in the $D_n$-type number system. We define the subexceedant function $f$ on the set $I_n$ as follows: 
\begin{equation}\label{subexc}
f(1)=1,~~ f(i)=1+\lfloor{\frac{d_{i-1}}{2}}\rfloor,~~\textrm{for}~~\textrm{all}~~2 \leq i \leq n.
\end{equation}
\end{defn}
It is clear here that $1 \leq f(i) \leq i,~~\textrm{for}~~\textrm{all}~~1 \leq i \leq n$.
We define the function by the rule $\tau(x):=$the number of odd components appearing in the expression $x=(d_{n-1}:d_{n-2}: \cdots : d_2: d_1)_{D_n}$.
Having defined the sign $s_i=(-1)^{d_{i-1}}$ for all $2 \leq i \leq n$ and taken the sign $s_1=(-1)^{\tau(x)}$, we associate each $x=(d_{n-1}:d_{n-2}: \cdots : d_2: d_1)_{D_n}$ in the $D_n$-type number system to a unique even-signed permutation 
$$\alpha_x=[s_1 \beta_1, \cdots, s_n \beta_n ],$$
where $\beta_f=[\beta_1, \cdots, \beta_n]$ is the permutation of $S_n$ that is the image $\phi(f)$ of the subexceedant function $f$ under $\phi$ given in equation (\ref{subex}).

Thus, we map each positive integer $x$ given in the $D_n$-type number system to an element of the group of even-signed permutations. Conversely, we will now show how to associate any element of this group to a positive integer. For this purpose, we take any even-signed permutation $\pi=[s_1 \gamma_1, \cdots, s_n \gamma_n ]$, where $\gamma \in S_n$. First of all we determine the subexceedent function $f$ in relation to $\pi$ in the following manner:
\begin{enumerate}
    \item let $f=\phi^{-1}(\gamma) \in \mathcal{F}_n$
    \item \textrm{for}~~\textrm{all}~~$2 \leq i \leq n$, define $r_i = \left\{
\begin{array}{ll}
      0 & s_{i+1} > 0 \\
      1 & s_{i+1} < 0 \\
\end{array} 
\right.$
 \item set $d_i=2(f(i+1)-1)+r_i, \textrm{for}~~\textrm{all}~~1 \leq i \leq n-1$
 \item establish $x=(d_{n-1}:d_{n-2}: \cdots : d_2: d_1)_{D_n}$.
\end{enumerate}
By checking the sign $s_1$ in the even-signed permutation $\pi$, it can be verified that the number of odd integer components involving in the expression of $x$ in ${D_n}$-type base system is odd or even. As a result of the above facts, we can state the following theorem without proof.

\begin{thm}
There is one-to-one correspondence between positive integers and elements of the group of even-signed permutations.
\end{thm}

\begin{ex}
When consider $x=151100130419=(3:15:6:9:8:5:4:5:7:2:3)_{D_{12}}$ we determine the subexceedant function depending on the equation(\ref{subexc}) as $f=f(1);f(2);f(3);f(4);f(5);f(6);f(7);f(8);f(9);f(10);f(11);\\f(12);=1;2;2;4;3;3;3;5;5;4;8;2$. Since $\tau(x)=7$ and so we get $\alpha_x=[-1,-11,\\-12,10,-6,7,-3,9,-5,-4,8,-2 ] \in D_{12}.$
\end{ex}

\begin{ex}
Let $\pi=[4,3,8,12,-9,-7,-10,-11,1,5,-2,-6 ] \in D_{12}$. We obtain the subexceedant function associated with $\pi$ as $f=f(1);f(2);f(3);f(4);\\f(5);f(6);f(7);f(8);f(9);f(10);f(11);f(12);=1;2;2;1;1;5;5;2;1;5;2;6$. Hence we get the integer representation of $\pi$ as
$$455941042762=(11:3:8:0:3:9:9:1:0:2:2)_{D_{12}}.$$
\end{ex}

\end{document}